\documentclass[%
aip,
 sd,%
 amsmath,amssymb,
preprint,%
]{revtex4-1}
\usepackage{amssymb}
\usepackage{amsthm}
\usepackage{graphicx}
\usepackage{epstopdf}
\usepackage{dcolumn}
\usepackage{bm}

\usepackage{enumerate}
\newtheorem{theorem}{Theorem}[section]

\newtheorem{definition}{Definition}[section]

\begin{document}

\preprint{AIP/123-QED}

\title{An averaging principle for fractional stochastic differential equations  with  L\'{e}vy noise}

\author{Wenjing Xu}
 \altaffiliation{xuwenjing9121@163.com.}
 \affiliation{School of Science, Northwestern Polytechnical University, Xi'an 710129, China
}%

\author{Jinqiao Duan}
 \altaffiliation{duan@iit.edu}
\affiliation{Department of Applied Mathematics, Illinois Institute of
Technology, Chicago 60616, USA
}%

\author{Wei Xu}%
 \altaffiliation{Author to whom correspondence should be addressed. Electronic mail: weixu@nwpu.edu.cn}
 \affiliation{School of Science, Northwestern Polytechnical University, Xi'an 710129, China
}%

\date{\today}

\begin{abstract}
This paper is devoted to the study of an averaging principle for fractional stochastic differential equations in $\mathbb{R}^{n}$ with L\'{e}vy motion, using an integral transform method.
We obtain a time-averaged equation under suitable assumptions.
Furthermore, we show that the solutions of averaged equation   approach the solutions of  the original equation.
Our results  in this paper provide  better understanding  for effective approximation of fractional dynamical systems with non-Gaussian L\'{e}vy noise.

\end{abstract}


\maketitle

\begin{quotation}
Fractional stochastic differential equations are alternative models for anomalous dynamics  in  various complex systems under non-Gaussian random fluctuations.
Although stochastic averaging methods are widely used to gain macroscopic dynamics for multiscale  stochastic differential equations, it is still a great challenge to derive effective models to approximate \emph{fractional} stochastic differential equations, letting alone the systems with non-Gaussian L\'{e}vy motion.
We take up the challenge to study a fractional averaging principle for a dynamical system with L\'{e}vy motion and provide a theoretical foundation of fractional stochastic averaging methods.
This offers a reduced yet effective way to accurately predict the solution paths of fractional stochastic systems with L\'{e}vy motion, under suitable conditions.


\end{quotation}

\section{Introduction}
This paper is devoted to the averaging theory of Caputo type fractional stochastic equations with L\'{e}vy motion, looking for an averaged equation in the mean square convergence sense.
The theory of averaging are indispensable \cite{fr2, bo1, fm221} that they provide evidences or justifications for the averaging procedures of complex equations arising from mathematics, control, engineering mechanics and several problems.
These theoretical results play a crucial role in investigating perturbation theory and  nonlinear dynamical systems during their long history.

Stochastic averaging principles, as a kind of effective analysis tool, are presented to help us approach stochastic differential equations (SDEs) with various different noises, such as multiplicative noise \cite{cs3}, Poisson noise \cite{zwq4, zwq5}, fractional Brownian motion \cite{xy6, xy7}, general stochastic measure\cite{vr8} and the like. L\'{e}vy noise is an important non-Gaussian noise, and the enthusiasm of researchers are growing in this filed. Recently, Xu in a literature \cite{xy9} gave the integer averaging principle with L\'{e}vy noise. However, until now, the fractional averaging behavior with L\'{e}vy motion is not well understood.

Fractional derivative, which can characterize the memory and hereditary properties of various practical dynamical systems, has been widespread concerned \cite{kv91, kv92, kv93}.
However, most scholars prefer integer order derivative to fractional derivative in the concrete research process, because there is no rigorous mathematic tool.
With the in-depth research, they gradually find that fractional averaging methods can more easily help us simplify and obtain approximate solutions to fractional nonclassical dynamical systems.
Roughly speaking, the theory of fractional averaging provides us ample opportunity to accurately reveal the essence of real life \cite{sc141}.

Fractional averaging principle is in the innovation phase.
In our previous papers \cite{xwj10, xwj11}, by analyses of solutions before and after averaging, we have proved that averaging principles are satisfied for both Caputo fractional stochastic differential equations with Brownian motion and fractional neutral equations with Poisson jumps. Here, we consider the following fractional stochastic equations,
\begin{equation}\label{1}
\left\{
   \begin{aligned}
 D_{t}^{\beta}X\left(t\right) &=b\left(t,X\left(t\right)\right)+\sigma\left(t,X\left(t\right)\right)\frac{dL\left(t\right)}{dt},\\
 X\left(0\right)  &=X_{0},
   \end{aligned}
\right.
\end{equation}
where $D_{t}^{\beta}$ is the Caputo fractional derivative, $\beta\in(\frac{1}{2},1)$, initial value $\mathbb{E}\left| X_{0}\right|^{2}<\infty$,
functions $b:\left[0, T\right]\times \mathbb{R}^{n}\rightarrow\mathbb{R}^{n}$ and   $\sigma:\left[0, T\right]\times \mathbb{R}^{n}\rightarrow\mathbb{R}^{n\times r}$ are measurable, $L\left(t\right)$ is a $r$-dimensional L\'{e}vy motion.
The above type of equations are of significant importance in applications \cite{rk160}, which appear in many problems, such as high variability and long memory signal models \cite{sm161}, subsurface solute transport \cite{ls16, lh160}, infrared remote sensing \cite{ha162}, among others.
They are the ones which attract so much attention since the beginning.
We consider in this paper is the averaging character of these substantial interesting equations.

We would like to highlight the fact that our work here is motivated by Xu et al \cite{xy9},  who studied the averaging principle with L\'{e}vy noise.
In this article, we shall generalize the classical Khasminskii averaging approach to Caputo type fractional stochastic equations with L\'{e}vy motion.


We first recall some essential definitions and
existing results (Section 2).  Then we present the fractional averaging results, establishing a stochastic averaging principle for Caputo type fractional equations with L\'{e}vy motion (Section 3). Finally, the example is discussed to illustrate the main results (Section 4), and conclusion is given (Section 5).

\section{Definitions and existing results}
\label{sec.2}

Before researching our averaging results, let us give some basic concepts about fractional calculus and L\'{e}vy motion.

\begin{definition}\cite{kilb}
{\rm Let $f$ be a Lebesgue integrable function, $\frac{1}{2}<\beta<1$, the $\beta$-order integral is defined by
\begin{equation*}
I^{\beta}f(t)=\frac{1}{\Gamma(\beta)}\int^{t}_{t_{0}}(t-s)^{\beta-1}f(s)ds,~ t\in[t_{0},\infty),
\end{equation*}
where $\Gamma$ is the gamma function.}
\end{definition}

\begin{definition}\cite{kilb}
{\rm Let $\frac{1}{2}<\beta<1$, the $\beta$-order Caputo derivative for the function $f$ is
\begin{equation*}
D^{\beta}_{t}f\left(t\right)=\frac{1}{\Gamma(1-\beta)}\int^{t}_{t_{0}}\frac{f'(s)}{(t-s)^{\beta}}ds,~ t\in[t_{0},\infty).
\end{equation*}
}
\end{definition}

\begin{theorem}\cite{ad13, djq14}
{\rm
If L\'{e}vy motion $L\left(t\right)$ is in $\mathbb{R}^{r}$, then the expression
\begin{equation*}
L\left(t\right)=mt+B(t)+\int_{\mid x\mid<c}x\widetilde{N}(t,dx)++\int_{\mid x\mid\geq c}xN(t,dx)
\end{equation*}
is called the L\'{e}vy-It\^{o} decomposition, where vector $m\in\mathbb{R}^{r}$, constant $c>0$, $r$-dimensional Brownian motion $B(t)$ has the covariance matrix $A$, $N(t,dx): \mathbb{R}^{+}\times \{\mathbb{R}^{r}-0\}$ controlling small jumps, is Poisson random measure, $\widetilde{N}(t,dx)=N(t,dx)-t\nu(dx)$ controlling large jumps, is compensated Poisson random measure, $\nu$ is the jump measure.
}
\end{theorem}

Using the above theorem, let's rewrite equation (\ref{1}) and give a more general representation
\begin{equation*}
\left\{
   \begin{aligned}
 D_{t}^{\beta}X\left(t\right) =&f\left(t,X\left(t-\right)\right)+G\left(t,X\left(t-\right)\right)\frac{dB\left(t\right)}{dt}+
\frac{1}{dt}\int_{\mid x\mid<c}H\left(t,X\left(t-\right),x\right)\widetilde{N}(dt,dx)\\
   &+\frac{1}{dt}\int_{\mid x\mid\geq c}Q\left(t,X\left(t-\right),x\right)N(dt,dx),\\
 X\left(0\right)  =&X_{0},
   \end{aligned}
\right.
\end{equation*}
where functions $f:\left[0, T\right]\times \mathbb{R}^{n}\rightarrow\mathbb{R}^{n}$,  $G:\left[0, T\right]\times \mathbb{R}^{n}\rightarrow\mathbb{R}^{n\times r}$,
$H:\left[0, T\right]\times \mathbb{R}^{n}\times \mathbb{R}^{n}
\rightarrow\mathbb{R}^{n}$
and
$Q:\left[0, T\right]\times \mathbb{R}^{n}\times \mathbb{R}^{n}
\rightarrow\mathbb{R}^{n}$ are measurable.
According to the technique presented in the literature \cite{ad13}, we set our sights on the L\'{e}vy motion without large jumps,
\begin{equation}\label{2}
\left\{
   \begin{aligned}
 D_{t}^{\beta}X\left(t\right) =&f\left(t,X\left(t-\right)\right)+G\left(t,X\left(t-\right)\right)\frac{dB\left(t\right)}{dt}+\\
&\frac{1}{dt}\int_{\mid x\mid<c}H\left(t,X\left(t-\right),x\right)\widetilde{N}(dt,dx),\\
 X\left(0\right)  =&X_{0}.
   \end{aligned}
\right.
\end{equation}

Let us make two assumptions on functions $f$, $G$ and $H$,
\begin{enumerate}
\item[$\left(H_{1}\right)$]\label{a} \ Let $x_{1},x_{2}\in\mathbb{R}^{n}$, $t\in\left[0, T\right]$ and  constant $C_{1}>0$. Then
   \begin{equation*}
    \begin{split}
       &\left| f\left(t,x_{1}\right)-f\left(t,x_{2}\right) \right|^{2} \vee
        \| a\left(t,x_{1},x_{1}\right)-2a\left(t,x_{1},x_{2}\right)+a\left(t,x_{2},x_{2}\right) \| \vee  \\
       &\int_{\mid x\mid<c}\left|H\left(t,x_{1},x\right)-H\left(t,x_{2},x\right)\right|^{2}\nu(dx)
       \leq C_{1}\left| x_{1}-x_{2}\right|^{2},
    \end{split}
    \end{equation*}
where $\mid\cdot\mid$ is $\mathbb{R}^{n}$-form, $\|\cdot\|$ is matrix form, $x_{1} \vee x_{2}=max\{x_{1}, x_{2}\}$ and  $a\left(t,x_{1},x_{2}\right)=G(t,x_{1})G(t,x_{2})'$ is a $n\times n$ matrix.

    \
\end{enumerate}

\begin{enumerate}
\item[$\left(H_{2}\right)$]\label{b} \  Let $x_{1}\in\mathbb{R}^{n}$, $t\in\left[0, T\right]$ and constant $C_{2}>0$. Then
    \begin{equation*}
       \left| f\left(t,x_{1}\right) \right|^{2} \vee
       \left\| a\left(t,x_{1},x_{1}\right) \right\|  \vee
       \int_{\mid x\mid<c}\left|H\left(t,x_{1},x\right)\right|^{2}\nu(dx)
       \leq C_{2}(1+\left| x_{1}\right|^{2}).
    \end{equation*}

    \
\end{enumerate}

Conduct assumptions $\left(H_{1}\right)$ and $\left(H_{2}\right)$, equation (\ref{2}) have the unique, adapted and cadlag mild solution
\begin{equation}\label{3}
\begin{split}
X\left(t\right)=&X_{0}+\frac{1}{\Gamma(\beta)}\int^{t}_{0}\left(t-s\right)^{\beta-1}f\left(s,X\left(s-\right)\right)ds+\\
                &\frac{1}{\Gamma(\beta)}\int^{t}_{0}\left(t-s\right)^{\beta-1}G\left(s,X\left(s-\right)\right)dB(s)+\\
                &\frac{1}{\Gamma(\beta)}\int^{t}_{0}\left(t-s\right)^{\beta-1}
                \int_{\mid x\mid<c}H\left(s,X\left(s-\right),x\right)\widetilde{N}(ds,dx),
\end{split}
\end{equation}
where $\mathbb{E}(  \int^{T}_{0} \left|X\left(t\right)\right|^{2} dt )<\infty$.
See \cite{xy9, kilb, ad13, djq14, zxm15} for more details.

\section{Averaging principle for fractional stochastic equations}

In previous sections, all tools needed for the averaging problem of Caputo type fractional stochastic equations with L\'{e}vy noise have been prepared.
Start this section by giving the standard form of equation (\ref{2}):
\begin{equation}\label{4}
\left\{
   \begin{aligned}
 D_{t}^{\beta}X_{\epsilon}\left(t\right) =
  &\epsilon f\left(t,X_{\epsilon}\left(t-\right)\right)+
    \sqrt{\epsilon}G\left(t,X_{\epsilon}\left(t-\right)\right)\frac{dB\left(t\right)}{dt}+\\
  &\frac{\sqrt{\epsilon}}{dt}\int_{\mid x\mid<c}H\left(t,X_{\epsilon}\left(t-\right),x\right)\widetilde{N}(dt,dx),\\
 X_{\epsilon}\left(0\right)  =&X_{0},
   \end{aligned}
\right.
\end{equation}
which is obtained by some time scale transformations. Here, $\epsilon<<1$ is a small positive parameter in $\left(0, \epsilon_{0}\right]$.

Taking the average of functions $f$, $G$, $H$ with respect to $t$, we are going to show that the solutions
of equation (\ref{4}) can be approached by the solutions of time-averaged equation
\begin{equation}\label{5}
\left\{
   \begin{aligned}
 D_{t}^{\beta}Z_{\epsilon}\left(t\right) =
  &\epsilon \overline{f}\left(Z_{\epsilon}\left(t-\right)\right)+
    \sqrt{\epsilon} ~\overline{G}\left(Z_{\epsilon}\left(t-\right)\right)\frac{dB\left(t\right)}{dt}+\\
  &\frac{\sqrt{\epsilon}}{dt}\int_{\mid x\mid<c}\overline{H}\left(Z_{\epsilon}\left(t-\right),x\right)\widetilde{N}(dt,dx),\\
 Z_{\epsilon}\left(0\right)  =&X_{0},
   \end{aligned}
\right.
\end{equation}
where functions $\overline{f}$, $\overline{G}$, $\overline{H}$ satisfying
\begin{enumerate}
\item[$\left(H_{3}\right)$]\label{c} \ Let $x_{1}\in \mathbb{R}^{n}$, $T_{1}\in\left[0,T\right]$ and $\alpha_{i}\left(T_{1}\right)>0,~ i=1, 2, 3$.
Then
\begin{equation*}
\left| f\left(T_{1}, x_{1}\right)-\overline{f}\left(x_{1}\right)\right|
 \leq \alpha_{1}\left(T_{1}\right)\left(1+\left| x_{1}\right| \right),
\end{equation*}

\begin{equation*}
\left\|a\left(T_{1},x_{1},x_{1}\right)-\overline{a}\left(x_{1},x_{1}\right)\right\|
 \leq \alpha_{2}\left(T_{1}\right)\left(1+\left| x_{1}\right|^{2} \right),\\
\end{equation*}
and \\
\begin{equation*}
\int_{\mid x\mid<c}\left|H\left(T_{1},x_{1},x\right)-H\left(x_{1},x\right)\right|^{2}\nu(dx)
 \leq \alpha_{3}\left(T_{1}\right)\left(1+\left| x_{1}\right|^{2} \right),
\end{equation*}
where $\mathop {\lim }\limits_{{T_1} \to \infty } {\alpha _i}\left({T_1}\right) = 0$.

 \
\end{enumerate}

%

\begin{theorem}
{\rm
Let $\delta_{1} > 0$ and suppose assumptions $\left(H_{1}\right)$-$\left(H_{3}\right)$ hold for functions $f$, $G$ and $H$. Then there exist three constants $L > 0$, $\epsilon_{1}\in\left(0,\epsilon_{0}\right]$ and $\lambda\in \left(0, 1\right)$ such that
\begin{equation*}
\mathbb{E}( \underset{t\in [0, L\epsilon^{-\lambda}]}{\mathop{\sup }}\ \left| X_{\epsilon}\left(t\right)-Z_{\epsilon}\left(t\right)\right|^{2} )\leq\delta_{1}
\end{equation*}
for all $\epsilon\in\left(0,\epsilon_{1}\right]$.
}
\end{theorem}

\begin{proof}
{\rm
Take a value from $\left(0, T\right]$ and define as $u$.
Then for any $t\in \left[0, u\right]$
\begin{equation*}
\begin{split}
&X_{\epsilon}\left(t\right)-Z_{\epsilon}\left(t\right)\\
=&\frac{\epsilon}{\Gamma(\beta)}\int^{t}_{0}\left(t-s\right)^{\beta-1}
      \left[ f\left(s,X_{\epsilon}\left(s-\right)\right)-\overline{f}\left(Z_{\epsilon}\left(s-\right)\right) \right] ds+\\
     &\frac{\sqrt{\epsilon}}{\Gamma(\beta)}\int^{t}_{0}\left(t-s\right)^{\beta-1}\left[G\left(s,X_{\epsilon}\left(s-\right)\right)-\overline{G}\left(Z_{\epsilon}\left(s-\right)\right)\right]dB(s)+\\
     &\frac{\sqrt{\epsilon}}{\Gamma(\beta)}\int^{t}_{0}\left(t-s\right)^{\beta-1}
                \int_{\mid x\mid<c}\left[H\left(s,Z_{\epsilon}\left(s-\right),x\right)-\overline{H}\left(Z_{\epsilon}\left(s-\right),x\right)\right]\widetilde{N}(ds,dx).
\end{split}
\end{equation*}
Computing the mathematical expectation, we have
\begin{equation}\label{6}
\begin{split}
&\mathbb{E} (\underset{0\le t\le u}{\mathop{\sup }}\
\left| X_{\epsilon}\left(t\right)-Z_{\epsilon}\left(t\right) \right|^{2} )\\
    \leq
    &\frac{3\epsilon^{2}}{\Gamma(\beta)^{2}}
    \mathbb{E} \underset{0\le t\le u}{\mathop{\sup}}\
      \left|\int^{t}_{0}\left(t-s\right)^{\beta-1}
      \left[f\left(s,X_{\epsilon}\left(s-\right)\right)-\overline{f}\left(Z_{\epsilon}\left(s-\right)\right) \right] ds \right|^{2}+\\
    &\frac{3\epsilon}{\Gamma(\beta)^{2}}
     \mathbb{E}\underset{0\le t\le u}{\mathop{\sup}}\ \left|\int^{t}_{0}\left(t-s\right)^{\beta-1}\left[ G\left(s,X_{\epsilon}\left(s-\right)\right)
      -\overline{G}\left(Z_{\epsilon}\left(s-\right)\right)\right]dB(s) \right|^{2}+\\
    &\frac{3\epsilon}{\Gamma(\beta)^{2}}
     \mathbb{E}\underset{0\le t\le u}{\mathop{\sup}}\
     \left| \int^{t}_{0}\left(t-s\right)^{\beta-1}
     \int_{\mid x\mid<c}\left[H\left(s,X_{\epsilon}\left(s-\right),x\right)-\overline{H}\left(Z_{\epsilon}\left(s-\right),x\right)\right]\widetilde{N}(ds,dx)\right|^{2}\\
     =&J_{1}+J_{2}+J_{3}.
\end{split}
\end{equation}
Evaluating $J_{1}$ with the technic of integration by parts and conditions $\left(H_{1}\right)$-$\left(H_{3}\right)$, produces
\begin{equation}\label{7}
\begin{split}
J_{1}
\leq    &\frac{6\epsilon^{2}}{\Gamma(\beta)^{2}} \mathbb{E} \underset{0\le t\le u}{\mathop{\sup}}\
        \left|\int^{t}_{0}\left(t-s\right)^{\beta-1}
        \left[ f\left(s,X_{\epsilon}\left(s-\right)\right)-f\left(s,Z_{\epsilon}\left(s-\right)\right)\right] ds
        \right|^{2}\\
        &+\frac{6\epsilon^{2}}{\Gamma(\beta)^{2}} \mathbb{E} \underset{0\le t\le u}{\mathop{\sup}}\
        \left|\int^{t}_{0}\left(t-s\right)^{\beta-1}
        \left[ f\left(s,Z_{\epsilon}\left(s-\right)\right)-\overline{f}\left(Z_{\epsilon}\left(s-\right)\right)\right] ds
        \right|^{2}\\
\leq    &\frac{6\epsilon^{2}u}{\Gamma(\beta)^{2}}
          \mathbb{E}\underset{0\le t\le u}{\mathop{\sup}}\
           \left(\int^{t}_{0}\left(t-s\right)^{2\beta-2}\left| f\left(s,X_{\epsilon}\left(s-\right)\right)
        -f\left(s,Z_{\epsilon}\left(s-\right)\right) \right|^{2}ds\right) \\
        &+\frac{6\epsilon^{2}}{\Gamma(\beta)^{2}}\mathbb{E}\underset{0\le t\le u}{\mathop{\sup}}\
        \left|\int^{t}_{0}
        \left[  f\left(s,Z_{\epsilon}\left(s\right)\right)
        -\overline{f}\left(Z_{\epsilon}\left(s\right)\right) \right]
        d\frac{-\left(t-s\right)^{\beta}}{\beta}
         \right|^{2}\\
\leq    &K_{11}\epsilon^{2}u \int^{u}_{0}\left(u-s\right)^{2\beta-2} \mathbb{E}(\underset{0\le s_{1}\le s}{\mathop{\sup}}\
        \left| X_{\epsilon}\left(s_{1}\right)-Z_{\epsilon}\left(s_{1}\right) \right|^{2})ds+
         K_{12} \epsilon^{2}u^{2\beta},
\end{split}
\end{equation}
where $K_{11}=\frac{6C_{1}^{2}}{\Gamma(\beta)^{2}} $ and $K_{12}=\frac{12}{\beta^{2}\Gamma(\beta)^{2}}
         \underset{0\le t\le u}{\mathop{\sup}}\  \alpha_{1}\left(t\right)^{2}
        [1+\mathbb{E}(\underset{0\le \tau\le u}{\mathop{\sup }}\ \left| Z_{\epsilon}(\tau)\right|^{2}) ]$.

Similarly, for $J_{2}$,
\begin{equation*}
\begin{split}
J_{2}
\leq &   \frac{6\epsilon}{\Gamma(\beta)^{2}}\mathbb{E}\underset{0\le t\le u}{\mathop{\sup}}\
        \left|\int^{t}_{0}\left(t-s\right)^{\beta-1}\left[ G\left(s,X_{\epsilon}\left(s\right)\right)
      -G\left(s,Z_{\epsilon}\left(s\right)\right) \right]dB(s) \right|^{2}+\\
     &  \frac{6\epsilon}{\Gamma(\beta)^{2}}\mathbb{E}\underset{0\le t\le u}{\mathop{\sup}}\
        \left|\int^{t}_{0}\left(t-s\right)^{\beta-1}\left[ G\left(s,Z_{\epsilon}\left(s\right)\right)
      -\overline{G}\left(Z_{\epsilon}\left(s\right)\right) \right]dB(s) \right|^{2},\\
\end{split}
\end{equation*}
adding Doob's martingale inequality and It$\hat{o}$ isometry,
\begin{equation}\label{8}
\begin{split}
J_{2}
\leq  &K_{21}\epsilon \int^{u}_{0}\left(u-s\right)^{2\beta-2}
       \mathbb{E}(\underset{0\le s_{1}\le s}{\mathop{\sup}}\
        \left| X_{\epsilon}\left(s_{1}\right)-Z_{\epsilon}\left(s_{1}\right) \right|^{2})ds + \\
      &\frac{6\epsilon}{\Gamma(\beta)^{2}}
          \mathbb{E}\underset{0\le t\le u}{\mathop{\sup}}\
       \int^{t}_{0} \left\|a\left(s,Z_{\epsilon}\left(s\right)\right)
        -\overline{a}\left(Z_{\epsilon}\left(s\right)\right)\right\|
        d \left[ \frac{-\left(t-s\right)^{2\beta-1}}{2\beta-1} \right]
        \\
\leq  &K_{21}\epsilon \int^{u}_{0}\left(u-s\right)^{2\beta-2}
       \mathbb{E}(\underset{0\le s_{1}\le s}{\mathop{\sup}}\
        \left| X_{\epsilon}\left(s_{1}\right)-Z_{\epsilon}\left(s_{1}\right) \right|^{2})ds +
         K_{22} \epsilon u^{2\beta-1},\\
\end{split}
\end{equation}
where $K_{21}=\frac{6C_{1}^{2}}{\Gamma(\beta)^{2}}$ and $K_{22}=\frac{6}{(2\beta-1)\Gamma(\beta)^{2}}
         \underset{0\le t\le u}{\mathop{\sup}}\  \alpha_{2}\left(t\right)
        [1+\mathbb{E}(\underset{0\le \tau\le u}{\mathop{\sup }}\ \left| Z_{\epsilon}(\tau)\right|^{2}) ]$.

In the sequel, for $J_{3}$,
\begin{equation*}
\begin{split}
J_{3}
\leq   &\frac{6\epsilon}{\Gamma(\beta)^{2}}
       \mathbb{E}
       \left( \int^{u}_{0}\left(u-s\right)^{2\beta-2}
      \int_{\mid x\mid<c}\left|H\left(s,X_{\epsilon}\left(s-\right),x\right)-
       H\left(s,Z_{\epsilon}\left(s-\right),x\right)\right|^{2}\nu(dx)ds\right)\\
     &+\frac{6\epsilon}{\Gamma(\beta)^{2}}
       \mathbb{E}
       \left( \int^{u}_{0}\left(u-s\right)^{2\beta-2}
      \int_{\mid x\mid<c}\left|H\left(s,X_{\epsilon}\left(s-\right),x\right)-
       \overline{H}\left(Z_{\epsilon}\left(s-\right),x\right)\right|^{2}\nu(dx)ds\right),\\
\end{split}
\end{equation*}
noting the assumptions and formula used above,
\begin{equation}\label{9}
\begin{split}
J_{3}
\leq   &K_{31}\epsilon \int^{u}_{0}\left(u-s\right)^{2\beta-2}
        \mathbb{E}(\underset{0\le s_{1}\le s}{\mathop{\sup}}\
         \left| X_{\epsilon}\left(s_{1}\right)-Z_{\epsilon}\left(s_{1}\right) \right|^{2})ds +\\
       &\frac{6\epsilon}{\Gamma(\beta)^{2}}
       \mathbb{E}
        \int^{u}_{0}
      \int_{\mid x\mid<c}\left|H\left(s,X_{\epsilon}\left(s-\right),x\right)-
       \overline{H}\left(Z_{\epsilon}\left(s-\right),x\right)\right|^{2}\nu(dx)
       d\left(\frac{-\left(u-s\right)^{2\beta-1}}{2\beta-1}\right)\\
\leq   &K_{31}\epsilon \int^{u}_{0}\left(u-s\right)^{2\beta-2}
        \mathbb{E}(\underset{0\le s_{1}\le s}{\mathop{\sup}}\
         \left| X_{\epsilon}\left(s_{1}\right)-Z_{\epsilon}\left(s_{1}\right) \right|^{2})ds +
          K_{32} \epsilon u^{2\beta-1},
\end{split}
\end{equation}
where $K_{31}=\frac{6C_{1}^{2}}{\Gamma(\beta)^{2}}$ and $K_{32}=\frac{6}{(2\beta-1)\Gamma(\beta)^{2}}
         \underset{0\le t\le u}{\mathop{\sup}}\  \alpha_{3}\left(t\right)
        [1+\mathbb{E}(\underset{0\le \tau\le u}{\mathop{\sup }}\ \left| Z_{\epsilon}(\tau)\right|^{2}) ]$.

Hence, we get

\begin{equation*}
\begin{split}
&\mathbb{E} (\underset{0\le t\le u}{\mathop{\sup }}\ \left| X_{\epsilon}\left(t\right)-Z_{\epsilon}\left(t\right) \right|^{2} )\\
\leq
&K_{12}\epsilon^{2}u^{2\beta} +
         (K_{22}+K_{32})\epsilon u^{2\beta-1}+
        \left (K_{11}\epsilon^{2}u+K_{21}\epsilon+K_{31}\epsilon \right)  \\ & \int^{u}_{0}\left(u-s\right)^{(2\beta-1)-1} \mathbb{E}(\underset{0\le s_{1}\le s}{\mathop{\sup}}\
        \left| X_{\epsilon}\left(s_{1}\right)-Z_{\epsilon}\left(s_{1}\right) \right|^{2})ds,\\
\end{split}
\end{equation*}
moreover \cite{wq16, hy17},
\begin{equation*}
\begin{split}
&\mathbb{E} (\underset{0\le t\le u}{\mathop{\sup }}\ \left| X_{\epsilon}\left(t\right)-Z_{\epsilon}\left(t\right) \right|^{2} )\\
\leq    &\left(K_{12}\epsilon^{2}u^{2\beta} + (K_{22}+K_{32})\epsilon u^{2\beta-1} \right) \\&
         \sum_{k=0}^{\infty}\frac{ \left[ \left (K_{11}\epsilon^{2}u^{1+\beta}+(K_{21}+K_{31})\epsilon u^{\beta} \right)\Gamma\left(\beta\right) \right]^{k} }
         {\Gamma\left(k\beta+1\right)}.  \\
\end{split}
\end{equation*}
Thus, we can find $L>0$ and $\lambda\in\left( 0,1 \right)$ such that for every $t\in\left(0, L\epsilon^{-\lambda}\right]\subseteq\left[0,T\right]$ having
\begin{equation*}
\begin{split}
&\mathbb{E} (\underset{0\le t\le L\epsilon^{-\lambda}}{\mathop{\sup }}\ \left| X_{\epsilon}\left(t\right)-Z_{\epsilon}\left(t\right) \right|^{2} )\leq C\epsilon^{1-\lambda},
\end{split}
\end{equation*}
where
\begin{equation*}
\begin{split}
C=      &\left(K_{12}L^{2\beta}\epsilon^{1+\lambda-2\beta\lambda} + (K_{22}+K_{32})L^{2\beta-1}\epsilon^{2\lambda(1-\beta)} \right)\\
        & \sum_{k=0}^{\infty}\frac{ \left[ \left (K_{11}L^{1+\beta}\epsilon^{2-\lambda-\beta\lambda}+(K_{21}+K_{31})L^{\beta}\epsilon ^{1-\beta\lambda} \right)\Gamma\left(\beta\right) \right]^{k} }
         {\Gamma\left(k\beta+1\right)}\\
\end{split}
\end{equation*}
is a constant.

}
\end{proof}

\section{Example}
In this section, we reduce a Caputo type fractional stochastic equations into simpler form to illustrate our main result.

Consider
\begin{equation}\label{10}
\left\{
   \begin{aligned}
D_{t}^{\beta}X_{\epsilon}\left(t\right)
=   &2\epsilon X_{\epsilon}cos^{2}\left(t\right) +
      \sqrt{\epsilon} \frac{dB\left(t\right)}{dt} +
       \sqrt{\epsilon}\int_{\mid x\mid<c}2x^{4}sin\left(t\right)^{2}X_{\epsilon}\left(t\right)\nu_{\alpha}(dx),\\
 X_{\epsilon}\left(0\right)=&0.1,
   \end{aligned}
\right.
\end{equation}
where $\beta\in(\frac{1}{2},1)$, $\alpha$-stable L\'{e}vy jump measure $\nu_{\alpha}(dx)=\frac{\gamma}{x^{1+\alpha}}dx$, constants $\gamma>0$, $\alpha\in(0,2)$.

Together with the averaged system
\begin{equation}\label{11}
   \begin{aligned}
 D_{t}^{\beta}Z_{\epsilon}\left(t\right) &=\epsilon Z_{\epsilon}(1+\gamma_{1}) +
                  \sqrt{\epsilon}\frac{dB\left(t\right)}{dt},
                   ~Z_{\epsilon}\left(0\right)=0.1,
                   ~\gamma_{1}=\frac{\gamma c^{4-\alpha}}{\sqrt{\epsilon}(4-\alpha)},
   \end{aligned}
\end{equation}
define
\begin{equation*}
Er=\left[ \left| X_{\epsilon}\left(t\right)-Z_{\epsilon}\left(t\right) \right|^{2} \right]^{\frac{1}{2}},
\end{equation*}
we numerically compare solution paths for equations (\ref{10}) and (\ref{11}) in Fig. 1.


%
%

\begin{center}
(a)\includegraphics[height=4cm,angle=0,width=5cm]{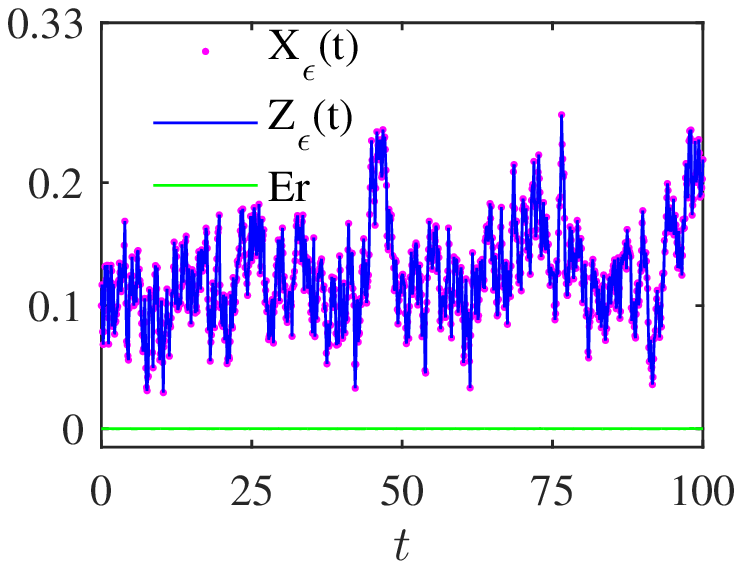}
(b)\includegraphics[height=4cm,angle=0,width=5cm]{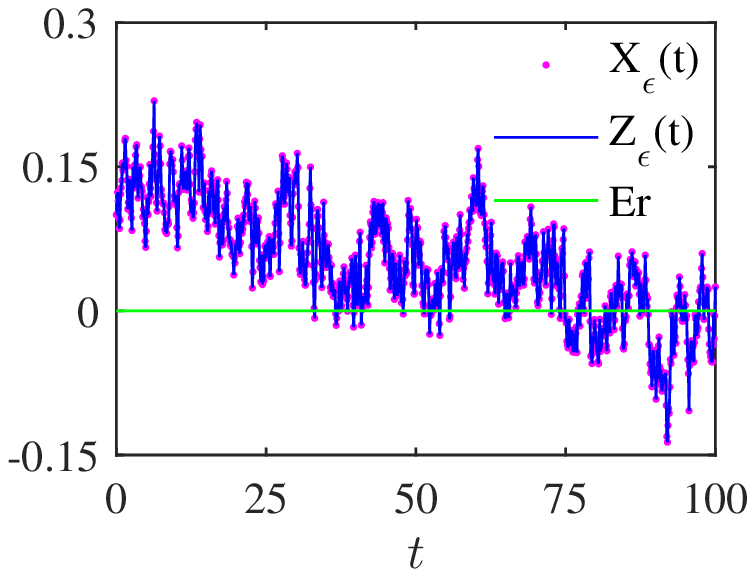}\\
(c)\includegraphics[height=4cm,angle=0,width=5cm]{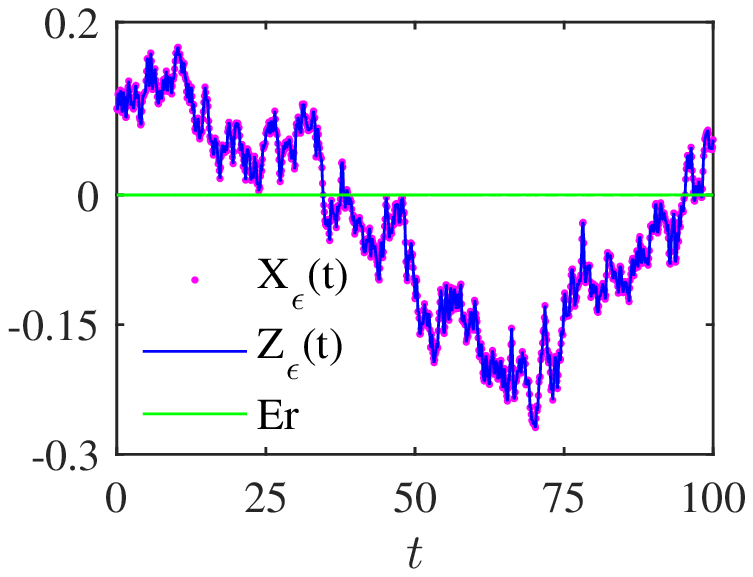}
(d)\includegraphics[height=4cm,angle=0,width=5cm]{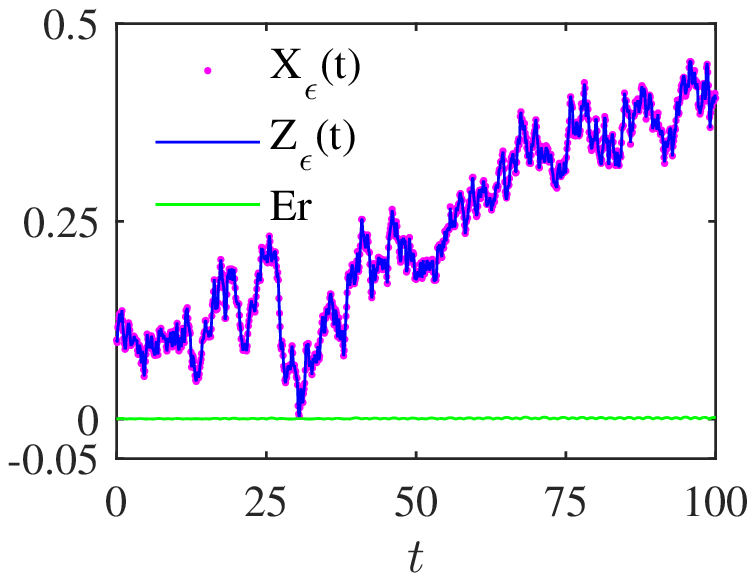}

 \bigskip
 \parbox[c]{12.5cm}{
Fig. 1. Solution paths of equations (\ref{10}) and (\ref{11}) when $\epsilon=0.001$, $c=0.5$:
 (a) $\beta=0.6$, $\alpha=0.3$, $\gamma=3$,
 (b) $\beta=0.6$, $\alpha=1.1$, $\gamma=0.6$,
 (c) $\beta=0.85$, $\alpha=0.3$, $\gamma=0.6$,
 (d) $\beta=0.85$, $\alpha=1.9$, $\gamma=3$.}
\end{center}

Note that a good agreement is demonstrated.

%

\section{Conclusion}
In this paper, we have proved an averaging principle for the Caputo type fractional stochastic equations with non-Gaussian L\'{e}vy motion.
It provides an effective stochastic approximation of solutions of fractional stochastic dynamical systems.

 As fractional stochastic differential equations arise as models for a variety of complex systems under non-Gaussian random influences, our method for fractional averaging will   be beneficial in extracting effective dynamical behaviors of such systems.


\section*{Acknowledgments}
This work was supported by the National Natural Science Foundation of China (No.11872305, 11532011) and China Scholarship Council (No.201906290182).

\section*{Data availability Statement}
The data that support the findings of this study are openly available in GitHub \cite{wjx28}.

\bibliographystyle{model1a-num-names}
\bibliography{<your-bib-database>}

\end{document}